\documentclass[11pt]{amsart}
\usepackage{amsfonts,amssymb,amsthm,amsmath,amsxtra,amscd,verbatim,eucal}

\usepackage[all]{xy}
\usepackage[dvips]{graphics}

\title{Multiplier ideals via Mather discrepancy}
\author{Lawrence Ein}
\author{Shihoko Ishii} 
\author{Mircea Musta\cedilla{t}\u{a}}
\address{Department of Mathematics, University of Illinois at Chicago, 
Chicago, IL 60607-7045, USA}

\email{ein@math.uic.edu}
\address{Department of Mathematical Science, University of Tokyo, Meguro, Tokyo, Japan}
\email{shihoko@ms.u-tokyo.ac.jp}
\address{Department of Mathematics, University of Michigan, Ann Arbor, MI 48109-1043, USA}
\email{mmustata@umich.edu}

\dedicatory{Dedicated to Professor Shigefumi Mori on the occasion of his 60th birthday}

\thanks{2010\,\emph{Mathematics Subject Classification}.
 Primary 14F18; Secondary 14B05.
\newline
The first author was partially supported by  NSF grant DMS-1001336,
the second author was partially supported by Grant-in-Aid (B) 22340004 and (S) 19104001,
and the third author was partially supported by NSF grant DMS-0758454, and
  by a Packard Fellowship}

\def\cL{\mathcal{L}}
\def\cO{\mathcal{O}}

\newcommand{\pd}{{\operatorname{pd}}}
\newcommand{\codim}{\operatorname{codim}}
\newcommand{\spec}{\operatorname{Spec}}

\let \cedilla =\c

\renewcommand{\o}[0]{{\mathcal O}} 

\newcommand{\hk}{{\widehat K}}
\renewcommand{\a}{{\frak{a}}}
\renewcommand{\b}{{\frak{b}}}

\renewcommand{\j}{{\mathcal J}}

\newcommand{\hj}{\widehat{\mathcal J}}
\newcommand{\hKY}{\widehat K_{Y/X}}

\def\to {\longrightarrow}

\newtheorem{thm}{Theorem}[section]

\newtheorem{lem}[thm]{Lemma}
\newtheorem{cor}[thm]{Corollary}
\newtheorem{prop}[thm]{Proposition}

\newtheorem{question}[thm]{Question}

\theoremstyle{definition}
\newtheorem{defn}[thm]{Definition}

\newtheorem{exmp}[thm]{Example}

\newtheorem{rem}[thm]{Remark}

\theoremstyle{remark}

\begin{document}
\maketitle

\begin{abstract}
 We define a version of multiplier ideals, the \emph{Mather multiplier ideals}, on a variety
 with arbitrary singularities, using
 the Mather discrepancy and the Jacobian ideal. In this context we prove
a relative vanishing theorem, thus obtaining restriction theorems and a subadditivity and summation theorems. 
The Mather multiplier ideals also satisfy a Skoda type result. As an application, we obtain a 
Brian\cedilla{c}on-Skoda type formula for the  integral closures  of ideals on a variety with arbitrary singularities. 
\end{abstract}

\section{Introduction}
The main goal of this paper is to introduce and prove the basic properties 
of \emph{Mather multiplier ideals}, a variant of the usual multiplier ideals which have found
wide applications recently in higher-dimensional algebraic geometry. Typically, the definition
of multiplier ideals is given in terms of a suitable resolution of singularities, and an ingredient that plays a fundamental role is the discrepancy (or relative canonical divisor). In order
for the discrepancy to be well-defined, one needs to make some assumptions 
on the singularities, namely to assume that the ambient variety is normal and 
${\mathbf Q}$-Gorenstein. The main advantage of Mather multiplier ideals is that they
can be defined on any variety (by which we mean an integral scheme of finite type over an algebraically closed field of characteristic zero). This agrees with the usual multiplier
ideals on normal, locally complete intersection varieties.

Given an arbitrary variety $X$ and a nonzero ideal $\a$ on $X$, we consider a log resolution
$f\colon Y\to X$ for the product of $\a$ with the Jacobian ideal ${\j ac}_X$ of $X$ (see \S 1 for the relevant definitions). Such a resolution has the property that the image of the canonical map
$f^*(\Omega_X^n)\to \Omega_Y^n$ (where $n=\dim(X)$) can be written as
$\cO_Y(-\widehat{K}_{Y/X})\cdot\Omega_Y^n$, for some effective divisor $\widehat{K}_{Y/X}$
on $Y$. This is the \emph{Mather discrepancy divisor}, and if we write 
${\j ac}_X\cdot\cO_Y=\cO_Y(-J_{Y/X})$,  then the difference $\widehat{K}_{Y/X}-J_{Y/X}$
plays the role that the discrepancy plays in the usual definition of multiplier ideals.
More precisely, if $\a\cdot\cO_Y=\cO_Y(-Z_{Y/X})$ and if $t\in {\mathbf R}_{\geq 0}$, then the
Mather multiplier ideal of $\a$ with exponent $t$ is given by
$$\hj(X,\a^t):=f_*\cO_Y(\widehat{K}_{Y/X}-J_{Y/X}-\lfloor t Z_{Y/X}\rfloor).$$
A priori this is only a fractional ideal, but we show that it is, in fact, contained in $\cO_X$.
If $X$ is normal and locally complete intersection, then the difference
$\widehat{K}_{Y/X}-J_{Y/X}$ is equal to the usual discrepancy $K_{Y/X}$, hence
the Mather multiplier ideals on $X$ agree with the usual multiplier ideals. 

The importance of multiplier ideals comes from the role they play in connection with vanishing theorems (see \cite[Chapter 9]{laz}). The basic result that accounts for this connection is 
a relative vanishing theorem. We prove a similar result in the context of Mather multiplier ideals. With the above notation, this says that we have the vanishing of all higher direct images:
$$R^if_*\cO_Y(\widehat{K}_{Y/X}-J_{Y/X}-\lfloor t Z_{Y/X}\rfloor)=0\,\,\text{for all}\,\,i>0.$$ 
As in the ``classical" case, this leads to some important properties of Mather multiplier ideals:
restriction theorems (we prove three such results, in various settings), a Subadditivity Theorem and a Summation Theorem, as well as a Skoda-type Theorem (see \cite{laz} for the
``classical" statements of these results). 
In proving these results, we make use of the approach developed by Eugene Eisenstein in
\cite{ee} for proving a version of the Restriction Theorem, and a Subadditivity Theorem for multiplier ideals on singular varieties.
We obtain as a corollary the following analogue of the 
Brian\c{c}on-Skoda Theorem on arbitrary varieties: if $\a$ is an ideal on $X$, and 
$\dim(X)=n$, then the integral closure of 
${\j ac}_X\cdot\a^m$ is contained in $\a^{m-n+1}$ for every $m\geq n$.

The notion of Mather discrepancy is suggested by the approach to singularities
via spaces of arcs (see \cite{e-Mus2}). This replacement for the usual discrepancy was first introduced in \cite{DEI}, where it was
applied to relate divisorial valuations to contact loci in arc spaces for singular varieties.
It was subsequently used in \cite{Ishii} to define new versions of some familiar invariants
of singularities (such as the log canonical threshold and the minimal log discrepancies)
on arbitrary varieties. 
While we were completing this manuscript, we found out
about the interesting preprint \cite{dd} of Tommaso de Fernex and Roi Docampo,
in which the authors also undertake a study of invariants of singularities defined from the point
of view of Mather discrepancies, and relate these to rational and du Bois singularities.

\noindent
{\bf Acknowledgement}
We are grateful to Takehiko Yasuda and Orlando Villamayor for useful conversations with the second  author on Nash blow-ups and factorizing resolutions, respectively.

\section{Definition of Mather multiplier ideal}

Recall that all our varieties are assumed to be irreducible and reduced, over a fixed algebraically 
closed field $k$ of characteristic zero. We begin by recalling the definition of Nash blow-up and 
Mather discrepancy, following \cite{DEI}. 

Let $X$ be an $n$-dimensional variety. The sheaf $\Omega_X^n=\wedge^n\Omega_X$ is invertible over 
the smooth locus $X_{\rm reg}$ of $X$, hence the morphism
$$\pi\colon {\mathbf P}(\Omega_X^n)={\mathcal Proj}({\rm Sym}(\Omega_X^n))\to X$$
is an isomorphism over $X_{\rm reg}$. The \emph{Nash blow-up} $\widehat{X}$ is the closure of
$\pi^{-1}(X_{\rm reg})$ in ${\mathbf P}(\Omega_X^n)$ (with the reduced scheme structure). Note that by construction
we have a projective birational morphism $\nu\colon\widehat{X}\to X$
and a surjective morphism $\nu^*(\Omega_X^n)\to \cO_{{\mathbf P}(\Omega_X^n)}(1)\vert_{\widehat{X}}$. Furthermore, this is universal in the following sense: given a birational morphism
of varieties $\phi\colon Y\to X$, this factors through $\nu$ if and only if 
$\phi^*(\Omega_X^n)$ admits a morphism onto an invertible sheaf $\cL$ on $Y$ (which 
identifies $\cL$ to the quotient of $\phi^*(\Omega_X^n)$ by its torsion).

Recall that a \emph{resolution of singularities} of $X$ is a projective birational morphism
$f\colon Y\to X$, with $Y$ smooth. This is a \emph{log resolution} of a nonzero ideal $\a$ on 
$X$ if 
$\a\cdot\cO_Y=\cO_Y(-D)$ for an effective divisor $D$, the exceptional locus of $f$ is an
effective divisor $E$,
and $D\cup E$ has simple normal crossings. Given several ideals 
$\a_1,\ldots,\a_r$, a log resolution of the product $\a_1\cdot\ldots\cdot \a_r$
gives a simultaneous log resolution of all $\a_i$, such that the divisors corresponding to each 
of the $\a_i$ and the exceptional divisor have simple normal crossings. 

\begin{defn} [\cite{DEI}]  Let $X$ be a variety of dimension $n$ and $f\colon Y\to X$ a resolution of singularities that factors through the Nash blow-up of $X$. 
In this case the image of the canonical homomorphism
$$f^*(\Omega^n_X) \to \Omega^n_Y$$
is an invertible sheaf of the form $J \cdot \Omega^n_Y$, 
where $J$ is the invertible ideal sheaf on $Y$ that defines an effective divisor supported on the exceptional locus of $f$. 
This divisor is called the \emph{Mather discrepancy divisor} and we denote it by $\hKY$.
We refer to \cite{DEI} for further details about Nash blow-ups and Mather discrepancies.
\end{defn}

\begin{defn}
Recall that the \emph{Jacobian ideal} $\j ac_X$ of a variety $X$ is the $0^{\rm th}$ Fitting ideal 
${\rm Fitt}_0(\Omega_X)$ of $\Omega_X$. This is an ideal whose support is the singular locus of $X$.
If $f\colon Y\to X$ is a log resolution of $\j ac_X$, we denote by $J_{Y/X}$ the effective divisor
on $Y$ such that  $\j ac_X\cdot\o_Y=\o_Y(-J_{Y/X})$.
\end{defn}

\begin{rem}
\label{lipman}
 Every log-resolution of $\j ac_X$ factors through the Nash blow-up.
 Indeed, it follows from \cite[Lemma~1]{lip} that given
 a proper birational morphism $\varphi\colon Y\to X$ of varieties of dimension $n$,
 the following are equivalent:
  \begin{enumerate}
 \item
 The smallest Fitting ideal ${\rm Fitt}_0(\varphi^*\Omega_X)$ is invertible;
 \item
 The projective dimension $\pd(\varphi^*\Omega_X)\leq 1$ and $\varphi^*\Omega_X/
 (\varphi^*\Omega_X)^{tor}$ is locally free of rank $n$, where $(\varphi^*\Omega_X)^{tor}$
 is the torsion part of $\varphi^*\Omega_X$.
 \end{enumerate}
 If $\varphi$ is a log resolution of $\j ac_X$, then
 ${\rm Fitt}_0(\varphi^*\Omega_X)=\j ac_X\cdot\o_Y$ is invertible.
 Therefore by Lipman's result $\varphi^*\Omega_X/
 (\varphi^*\Omega_X)^{tor}$ is locally free of rank $n$, and we deduce that $\varphi$ factors
 through $\widehat{X}$ by the 
universal property of the Nash blow-up.
\end{rem}

\begin{defn}\label{Mather_multiplier}
Let $X$ be a variety, $\a\subseteq \o_X$  a nonzero ideal on $X$, and 
 $t\in{\mathbf R}_{\geq 0}$. Given a log resolution $f\colon Y\to X$ of $\j ac_X\cdot \a$,
we denote by $Z_{Y/X}$ the effective divisor on $Y$ such that
 $\a\cdot\o_Y=\o_Y(-Z_{Y/X})$.
The Mather multiplier ideal of $\a$ of exponent $t$ is defined by
$$\hj (X, \a^t)=f_*\left(\o_Y(\hk_{Y/X}-J_{Y/X}-\lfloor t Z_{Y/X}\rfloor)\right),$$
where for a real divisor $D$ we denote by $\lfloor D\rfloor$ the largest integral divisor $\leq D$.
\end{defn}

\begin{rem}
The Mather multiplier ideal does not depend on the choice of 
$f$. Indeed, arguing as in the proof of \cite[Theorem~9.2.18]{laz}, we see that it is enough to show that if $g\colon Y'\to Y$ is such that $f'=f\circ g$ is another log resolution of
$\j ac_X\cdot \a$, then
\begin{equation}\label{eq_indep_def}
\cO_Y(\hk_{Y/X}-J_{Y/X}-\lfloor t Z_{Y/X}\rfloor)=g_*\left(\o_{Y'}(\hk_{Y'/X}-J_{Y'/X}-
\lfloor t Z_{Y'/X}\rfloor)\right).
\end{equation}
Since $\hk_{Y'/X}=g^*(\hk_{Y/X})+K_{Y/X}$, where $K_{Y/X}$ is the (usual) discrepancy of
$h$, and $J_{Y'/X}=g^*(J_{Y/X})$ and $Z_{Y'/X}=g^*(Z_{Y/X})$, then (\ref{eq_indep_def})
follows from \cite[Lemma~9.2.19]{laz}, which is the main ingredient in showing that the usual definition of multiplier ideals is independent of log resolution.
\end{rem}
\begin{rem}
For nonzero ideals $\a_1,\ldots,\a_r$ on a variety $X$, one can similarly define a mixed
Mather multiplier ideal $\hj (X, \a_1^{t_1}\cdots\a_r^{t_r})$ for every $t_1,\ldots,t_r\in{\mathbf R}_{\geq 0}$. With the notation in Definition~\ref{Mather_multiplier}, if 
$f$ is a log resolution of ${\j ac}_X\cdot\a_1\cdot\ldots\cdot \a_r$, and if we put
$\a_i\cdot\cO_Y=\cO_Y(-Z_i)$, then
$$\hj (X, \a_1^{t_1}\cdots\a_r^{t_r})
=f_*\left(\o_Y(\hk_{Y/X}-J_{Y/X}-\lfloor t_1Z_1+\ldots+t_rZ_r\rfloor)\right).$$
For simplicity, we will mostly consider Mather multiplier ideals for one ideal, but all statements
have obvious generalizations to the mixed case.
\end{rem}
\begin{rem}
If $X$ is normal and locally a complete intersection, then $\hj (X, \a^t)=\j (X, \a^t)$, where the right hand side is the usual multiplier ideal (see \cite[Chapter 9]{laz} for definition).
Indeed, in this case the image of the canonical map $\Omega^n_X \to \omega_X$ is $\j ac_X\cdot\omega_X$,
hence 
$\hk_{Y/X}-J_{Y/X}=K_{Y/X}$.
In particular, we see that $\hj (X, \a^t)=\j (X, \a^t)$ if $X$ is smooth.
\end{rem}
\begin{rem}
For now, the Mather multiplier ideal is only a fractional ideal of $\o_X$, with
$\hj (X, \a^t)\subset i_*(\o_{X_{reg}})$, where $i\colon X_{reg}\hookrightarrow X$ is the inclusion of the smooth locus. This inclusion implies that
$\hj (X, \a^t)$ is an ideal of $\o_X$ if $X$ is normal. We will see in 
Corollary~\ref{cor_ideal} that, in fact, the same is true on every variety $X$.
\end{rem}

The assertions in the next proposition are an immediate consequence of the definition.
\begin{prop}\label{elementary}
Let $\a$ and $\b$ be nonzero ideals on the variety $X$.
\begin{enumerate}
\item If $\a\subseteq\b$, then $\hj(X,\a^t)\subseteq\hj(X,\b^t)$ for every 
$t\in {\mathbf R}_{\geq 0}$.
\item If $s\leq t$ are in ${\mathbf R}_{\geq 0}$, then $\hj(X,\a^t)\subseteq\hj(X,\a^s)$.
\item If $t\in {\mathbf R}_{\geq 0}$ and $m$ is a positive integer, then
$\hj(X, (\a^m)^t)=\hj(X,\a^{mt})$.
\end{enumerate}
\end{prop}

\begin{defn}
\label{fact}
  Let $A$ be a variety and $X$ a reduced closed subscheme of $A$.
  A morphism $\varphi_A\colon  \overline{A}\to A$ is a \emph{factorizing resolution} of $X$ inside $A$ if
  the following hold:
  \begin{enumerate}
  \item
  $\varphi_A$ is an isomorphism at the generic point of every irreducible component of $X$. 
  In particular, the strict transform $\overline{X}$ of $X$ is defined.
  \item The morphisms $\varphi_A$ and $\varphi_X:=\varphi_A|_{\overline X}$ are resolutions of singularities of $A$ and $X$, respectively, and the union of $\overline{X}$
  with the exceptional locus ${\rm Exc}(\varphi_A)$ has simple normal crossings.
  \item If $I_X$ and $I_{\overline X}$ are the defining ideals of $X$ and $\overline X$ 
  in $A$ and $\overline A$, respectively, then there exists an effective divisor 
  $R_{X/A}$ on $\overline A$ such that 
  $$I_X\cdot\o_{\overline A}=I_{\overline X}\cdot \o_{\overline A}(-R_{X/A}).$$  
  \end{enumerate}
\end{defn}

Note that the divisor $R_{X/A}$ in (3) is supported on ${\rm Exc}(\phi_A)$, hence by assumption 
has simple normal crossings with $\overline{X}$.
 Bravo and Villamayor proved in \cite{bv} that if $A$ is smooth and $X$ is reduced and equidimensional, then factorizing resolutions of $X$ in $A$ exist.
The following strengthening is due to Eisenstein \cite[Lemma 3.1]{ee}.

\begin{lem}
\label{singfac}
Let $X\subseteq A$ be a reduced closed subscheme of a variety $A$. Let $\varphi_1\colon
A'\to A$ be a birational morphism from a smooth variety $A'$ that is an isomorphism at the generic points of the irreducible components of $X$.
Let $X'$ be the strict transform of $X$ in $A'$, and let $E$ be a divisor on $A'$ with simple normal crossing support, such that no component of $X'$ is contained in $E$.
Then there exists a morphism $\varphi_2\colon\overline{A}\to A'$ such that $\varphi_A:=\varphi_1\circ\varphi_2$ is a factorizing resolution of $X$ inside $A$, and $\overline{X}\cup 
{\rm Exc}(\varphi_A)\cup
{\rm Supp} (\varphi_2^*E)$ has  simple normal crossings, 
where $\overline{X}$ is the strict transform of $X$ in $\overline A$.
\end{lem}

The following lemma is essentially \cite[Lemma~4.4]{ee}.

\begin{lem}
\label{adjunction}
Let $X\subseteq A$ be a reduced equidimensional closed subscheme of a smooth variety 
$A$, and let $\varphi_A\colon \overline A\to A$ be a factorizing resolution of $X$ inside $A$.
Let $c=\codim (X, A)$ and $I_X\cdot\o_{\overline A}=I_{\overline X}\cdot \o_{\overline A}(-R_{X/A})$.
If $\varphi_A|_{\overline X}$ factors through the  blow-up along the ideal 
$\j ac_X$, then
$$(K_{\overline{A}/A}- cR_{X/A})|_{\overline X}=\hk_{{\overline X}/X}-J_{{\overline X}/X}.$$
\end{lem}

\begin{prop}
\label{smoothrestriction}
 Let $X$  be a closed subvariety of the smooth variety $A$, with ${\rm codim}(X,A)=c$.
If $\tilde\a\subseteq \o_A$ is an ideal on $A$ whose restriction $\a=\tilde\a\cdot \o_X$  to $X$ is nonzero,
then for every $t\in{\mathbf R}_{\geq 0}$
$$\hj(X,\a^t)\subseteq \j(A, \tilde\a^t)\cdot\cO_X.$$
\end{prop}

\begin{proof}
Let $\b\subseteq \o_A$ be the pull-back  of $\j ac_X$ by the canonical surjection
$\o_A\to \o_X$. We begin by choosing a log resolution
$\varphi_1\colon A'\to A$ of $\b\cdot \tilde\a$.
Let $ E$ be the divisor on $A'$ such that $\o_{A'}(-E)=\tilde\a\b  \cdot\o_{A'}$.
Applying Lemma \ref{singfac}, we obtain a factorizing resolution $\varphi_A\colon
\overline A\to A$ of $X$ inside $A$, that is also a log resolution of $\b\cdot\tilde\a$.
Let $R_{X/A}$ be the divisor on $\overline A$ such that
$$I_X\cdot\o_{\overline A}=I_{\overline X}\cdot\o_{\overline A}(-R_{X/A}),$$
where $\overline{X}$ is the strict transform of $X$ in $\overline{A}$.
As $\varphi_A$ factors through $\varphi_1$, the morphism 
$\varphi_X:=\varphi_A|_{\overline X}$ is a log resolution of $\j ac_X$.
Let $\tilde\a\cdot \o_{{\overline A}}=\o_{{\overline A}}(-Z_{{\overline A}/A})$ and $D=K_{{\overline A}/A}-\lfloor t Z_{{\overline A}/A}\rfloor-cR_{X/A}$.
The exact sequence
$$0\to I_{\overline X}\cdot\o_{\overline A}(D)\to \o_{\overline A}(D)\to \o_{\overline X}(D\vert
_{\overline{X}})
\to 0$$
 yields the exact sequence
\begin{equation}
\label{exact} 
0\to{ \varphi_{A}}_*( I_{\overline X}\cdot\o_{\overline A}(D))\to {\varphi_{A}}_*( \o_{\overline A}(D))\to {\varphi_{X}}_*( \o_{\overline X}(D|_{\overline X}))
\end{equation}
$$ \to {R^1\varphi_{A}}_*( I_{\overline X}\cdot\o_{\overline A}(D))
.$$
We claim that the last term in this sequence is zero. 
Indeed, let $\psi\colon \tilde A\to \overline A$ be the blow-up along $\overline X$ 
with exceptional divisor $T$. 
If ${\mathcal F}=\o_{\tilde A}(-T+\psi^*D)$, then the projection formula gives $\psi_*{\mathcal F}=
 I_{\overline X}\cdot\o_{\overline A}(D)$.
For $p>0$, we have 
$$R^p\psi_*( \mathcal F)=\cO_{\overline{A}}(D)\otimes R^p\psi_*(\cO_{\tilde A}(-T))=0,\ \ 
   R^p(\varphi_{A}\circ \psi)_* ( \mathcal F)=0.$$
The second vanishing follows from the Local Vanishing Theorem for multiplier ideals
(see \cite[Theorem~9.4.1]{laz}), using the fact that if $I_X\cdot \cO_{\tilde{A}}=\cO_{\tilde{A}}(-F)$,
then we can write
$$-T+\psi^*(D)=K_{\tilde{A}/A}-\lfloor cF+t\psi^*(Z_{\overline{A}/A})\rfloor.$$
Using the Leray spectral sequence
   $$E_2^{p, q}={R^p\varphi_{A}}_*R^q\psi_*(\mathcal F)
  \Rightarrow R^{p+q}(\varphi_{A}\circ \psi)_* ( \mathcal F),$$
  we obtain $R^1\varphi_{A*}(\psi_*{\mathcal F})=0$, as claimed.
  
  Since $R_{X/A}$ is effective, we have by definition ${\varphi_A}_*( \o_{\overline A}(D))\subseteq \j (A, \tilde\a^t)$,
hence the exact sequence (\ref{exact}) gives
$$ {\varphi_X}_*( \o_{\overline X}(D|_{\overline X}))\subseteq \j (A, \tilde\a^t)\cdot\cO_X.$$
By Lemma~\ref{adjunction} we have 
$${\varphi_X}_*( \o_{\overline X}(D|_{\overline X}))={\varphi_X}_*\left( \o_{\overline X}(\hk_{{\overline X}/X}
-J_{{\overline X}/X}-\lfloor t Z_{\overline {X}/X}\rfloor \right)=\hj (X, \a^t),$$
which completes the proof of the proposition.
\end{proof}

\begin{cor}\label{cor_ideal} Given a variety $X$, a nonzero ideal $\a$ on $X$, and
$t\in{\mathbf R}_{\geq 0}$, the following hold:
\begin{enumerate}
\item  The Mather multiplier ideal $\hj(X,\a^t)$ is  an ideal of $\o_X$.
\item  The ideal $\hj(X,\a^t) $ is integrally closed.
\item If $\nu\colon X^{\nu}\to X$ is the normalization of $X$, then the integral closure 
$\overline{\j ac_X}$
  of $\j ac_X$ satisfies
  $$\overline{\j ac_X}\subseteq (\cO_X:  \nu_*\cO_{X^{\nu}}).$$
\end{enumerate}  
\end{cor}

\begin{proof}
(1) follows from Proposition \ref{smoothrestriction}, since $X$ can be locally embedded in a smooth variety, and the result is known (and easy to prove) for usual multiplier ideals
on smooth varieties.
For (2), it is enough to note that if $f\colon Y\to X$ is a birational morphism, with $Y$ normal,
and $E$
is a Weil divisor on $Y$ such that
${\mathcal I}:=
f_*\o_Y(E)$ is an ideal of $\cO_X$,
then ${\mathcal I}$ is integrally closed.

In order to show (3), let $\varphi\colon Y\to X$ be a log resolution of $\j ac_X$. We have the inclusions
$$\overline{\j ac_X}\subseteq \varphi_*\o_Y(-J_{Y/X})\subseteq \varphi_*\o_Y(\hk_{Y/X}-J_{Y/X})=\hj (X,\o_X)\subseteq \o_X.$$
Since $\varphi$
factors through $\nu$, it follows that $\varphi_*\o_Y(-J_{Y/X})$ is a $\nu_*\o_{X^{\nu}}$-module,
and we deduce the inclusion in (3).
\end{proof}

\begin{exmp}
  If $X=\spec k[x^2, x^3]$, then
  $$\hj(X,\o_X)=(x^2)\subset k[x^2,x^3].$$
  Note that $X$ is a curve with a cusp at the origin and has the normalization 
  $\nu\colon  \overline{X}=\spec k[x] \to X$.
  The morphism $\nu $ is a resolution of singularities and factors through the Nash blow-up
  (since the normalization of a curve factors through every blow-up).
  As $\nu^*\Omega_X$ is generated by $dx^2, dx^3$,
   the image of the canonical map
  $\nu^* \Omega_X\to \Omega_{\overline{X}}$
  is generated by $xdx$ and $x^2dx$, hence
  $\o_{\overline X}(-\hk_ {\overline X/X})=(x)$.
  On the other hand, $\j ac_X=(x^3,x^4)\subset k[x^2, x^3]$ and
  $\j ac_X\o_{\overline X}=(x^3)\subset k[x]$.
  Therefore $\o_{\overline X}(\hk_ {\overline X/X}-J_{\overline X/X})=(x^2)\subset k[x]$,
  so that  $\hj(X,\o_X)=(x^2)$, as an ideal of $k[x^2, x^3]$.
\end{exmp}

\section{The local vanishing theorem and applications}

We first recall that for every variety $X$ of dimension $n$ we have 
a canonical morphism $\eta_X\colon \Omega_X^n\to\omega_X$, where $\omega_X$ is the canonical sheaf
of $X$. This is clear when $X$ is normal, since in that case $\omega_X\simeq
i_*(\Omega_{X_{\rm reg}}^n)$, where $i\colon X_{\rm  reg}\hookrightarrow X$ is the inclusion of the smooth locus. In the general case, let $\nu\colon X'\to X$ be the normalization, so we 
have  $\eta_{X'}\colon \Omega_{X'}^n\to\omega_{X'}$. 
Since $\nu$ is finite and surjective, we have an isomorphism $\nu_*(\omega_{X'})
\simeq {\mathcal Hom}_{\cO_X}(\nu_*(\cO_{X'}),\omega_X)$, and the inclusion
$\cO_X\hookrightarrow \nu_*(\cO_{X'})$ induces a canonical morphism
$\nu_*(\omega_{X'})\to\omega_X$. We then obtain $\eta_X$ as the composition
$$\Omega^n_X\to \nu_*(\Omega_{X'}^n)\overset{\nu_*(\eta_{X'})}\to \nu_*(\omega_{X'})
\to\omega_X.$$

The following is a consequence of \cite[Proposition~9.1]{e-Mus2}.

\begin{lem}
\label{em}
  Let $X\subseteq {\mathbf A}^N$ be a closed subvariety of dimension $n$, defined by 
  the ideal $I_X=(f_1,\ldots, f_d)$, and let $c=N-n$.
  If the subscheme $M$ defined by $c$ functions $f_{i_1},\ldots, f_{i_c}$ 
has pure codimension $c$, and
$M=X$ at the generic point of $X$,
then there exists an isomorphism of $\omega_X$ with an ideal
of $\o_X$
  $$\omega_X\simeq \left( (I_M :I_X)+I_X\right)/I_X.$$
  Under this isomorphism, the image of the canonical map $\eta_X\colon \Omega^n_X\to 
  \omega_X$ corresponds to $\j ac_M|_X.$
\end{lem}

\begin{rem}
\label{general}
Suppose that $X\subseteq {\mathbf A}^N$ is as in Lemma~\ref{em}. If 
$F_i:=\sum_{j=1}^d a_{ij}f_j$, with $a_{ij}$ general elements of $k$, then 
$I_X=(F_1,\ldots, F_d) $ and every set of distinct $c$ elements $F_{i_1},\ldots, F_{i_c}$ 
defines a subscheme $M$ that satisfies the conditions in Lemma \ref{em}.
\end{rem}

\begin{lem}
\label{keylemma}
For every $n$-dimensional affine variety $X$,
there exist a non-zero regular function $h$ on $X$ and ideals $I, G\subseteq \o_X$
such that there is a surjective morphism $\Omega^n_X\to I$ and
we have
\begin{equation}\label{eq_keylemma}
h\cdot \j ac_X=I\cdot G.
\end{equation}
\end{lem}

\begin{proof}
We apply Lemma~\ref{em} using the notation in Remark~\ref{general}.
For every subset $J\subseteq\{1,\ldots,d\}$ with $c$ elements, we obtain
a subscheme $M_J$ of ${\mathbf A}^N$ defined by $(F_i\mid i\in J)$, that satisfies
the conditions in Lemma~\ref{em}. In particular, we obtain an isomorphism
$\alpha_J\colon\omega_X\to{\mathfrak q}_J\subseteq\cO_X$ that maps the image of
$\Omega_X^n\to \omega_X$ onto ${\j ac}_{M_J}\vert_X$. 

Let us fix a subset $J_0$ as above, so
$\alpha_{J_0}$ induces a surjective map from $\Omega_X^n$ onto 
$I={\j ac}_{M_{J_0}}\vert_X$. Note that the isomorphism $\alpha_J\circ
\alpha_{J_0}^{-1}\colon {\mathfrak q}_{J_0}\to {\mathfrak q}_J$ is given by $\alpha_J\circ\alpha_{J_0}(u)=
\frac{b_J}{a_J}u$ for some nonzero $a_J,b_J\in\cO(X)$ (if ${\mathfrak q}$ is an ideal in a domain $R$,
and $\varphi\colon {\mathfrak q}\to R$ is an $R$-linear map, then for every nonzero $u,v\in {\mathfrak q}$ we have
$\frac{\varphi(u)}{u}=\frac{\varphi(u)v}{uv}=\frac{\varphi(uv)}{uv}=\frac{u\phi(v)}{uv}=\frac{\varphi(v)}{v}$).
Furthermore, after multiplying each $a_J$ by a suitable factor,
we may assume that all $a_J$ are equal, and we denote the common value by $h$. 

In particular, by considering the image of $\Omega_X^n$, we deduce
$h\cdot {\j ac}_{M_J}\vert_X=b_J\cdot I$ for every $J$. Since
${\j ac}_X=\sum_J{\j ac}_{M_J}\vert_X$, we obtain (\ref{eq_keylemma}) by taking
$G=\sum_J(b_J)$. 
 \end{proof}
 
 \begin{rem}
 This lemma gives another proof for Remark \ref{lipman}.
 \end{rem}
 
 The following theorem gives an analogue for the Local Vanishing Theorem for multiplier ideals
 (see \cite[Theorem~9.4.1]{laz}) in the context of Mather multiplier ideals. 

\begin{thm} 
\label{vanishing}
Let $\a$ be a nonzero ideal on the variety $X$, and let $t\in{\mathbf R}_{\geq 0}$.
If  $f\colon Y\to X$ is a log-resolution of $\a\cdot\j ac_X$, and we put 
$\a\cdot \o_Y=\o_Y(-Z_{Y/X})$, then
$$R^if_*\left(\o_Y(\hk_{Y/X}-J_{Y/X}-\lfloor tZ_{Y/X}\rfloor)\right)=0\ \ \text{for all}\,\,\, i>0.$$
\end{thm}

\begin{proof}
Since the assertion is local on $X$, we may assume that $X$ is affine.
It follows from Remark~\ref{lipman} that we can factor $f$ through the Nash blow-up as
$Y\stackrel{\psi}\longrightarrow
 \widehat{X}\stackrel{\varphi}\longrightarrow X$.
Let $\o_{\widehat{X}}(1)=\o_{{\mathbf P}(\Omega^n_X)}(1)|_{\widehat{X}}$.
With $I$  as in Lemma \ref{keylemma},
we have a surjection  of graded algebras ${\rm Sym}(\Omega^n_X)\to \oplus_{i=0}^\infty I^i$.
This induces a closed immersion ${\rm Bl}_I(X)\hookrightarrow {\mathbf P}(\Omega_X^n)$,
hence an isomorphism ${\rm Bl}_I(X)\simeq \widehat{X}$ over $X$, which yields
$\o_{\widehat{X}}(1)\simeq I\cdot\o_{\widehat{X}}$.

It follows from \cite[Proposition~1.7]{DEI} that  $$\o_Y(\hk_{Y/X})\simeq\o_Y(K_Y)\otimes \psi^*(\o_{\widehat{X}}(1))^{-1}\simeq \o_Y(K_Y)\otimes (I\cdot\o_Y)^{-1}.$$ 
Since $f$ is a log resolution of ${\j ac}_X$, we deduce from Lemma~\ref{keylemma} 
that we may write $G\cdot\cO_Y=\cO_Y(-B)$ for a divisor $B$ on $Y$, and we have
$$\o_Y(\hk_{Y/X}-J_{Y/X}-\lfloor tZ_{Y/X}\rfloor)\simeq \o_Y(K_Y)\otimes (I\cdot\o_Y)^{-1}\otimes IG\cdot\o_Y\otimes
\o_Y(-\lfloor tZ_{Y/X}\rfloor)$$
$$\simeq \o_Y(K_Y-B-\lfloor tZ_{Y/X}\rfloor).$$
Since both $-B$ and $-Z_{Y/X}$ are relatively nef over $X$, the Kawamata-Viehweg Vanishing Theorem (see for example \cite[Theorem~1-2-3]{kmm}) implies the vanishing
in the theorem.
\end{proof}

\begin{rem}\label{rem_vanishing}
The vanishing in Theorem~\ref{vanishing} holds, with the same proof, for mixed Mather multiplier ideals. More precisely, if $\a_1,\ldots,\a_r$ are nonzero ideals on $X$, and
$f\colon Y\to X$ is a log resolution of ${\j ac}_X\cdot\a_1\cdot\ldots\cdot \a_r$ such that
$\a_i\cdot\cO_Y=\cO_Y(-Z_i)$, then 
$$R^if_*\left(\o_Y(\hk_{Y/X}-J_{Y/X}-\lfloor t_1Z_1+\ldots+t_rZ_r\rfloor)\right)=0$$
for every $t_1,\ldots,t_r\in{\mathbf R}_{\geq 0}$ and
every $i>0$.
\end{rem}

According to our principle  to replace the usual discrepancy $K_{Y/X}$ by $\hk_{Y/X}-J_{Y/X}$ we can define  ``canonical singularities" for a  not necessarily $\mathbf Q$-Gorenstein
variety.
The paper \cite{dd} introduces the notion of $J$-canonical singularities.
For an arbitrary variety $X$, we say that $X$ has $J$-canonical singularities if 
$\hk_{Y/X}-J_{Y/X}\geq 0$ for a log resolution $f:Y\to X$ of $\j ac_X$
(and automatically for all log resolutions of $\j ac_X$).
They prove the following statement (\cite[Theorem 7.5]{dd}), which we also obtain as a corollary of our results.

\begin{cor}
\label{j-canonical}
If a variety $X$ has $J$-canonical singularities, then $X$ is normal and has rational singularities.
\end{cor}

For the proof of the corollary, we use the following generalization of Fujita's vanishing 
theorem (the original Fujita's vanishing theorem is in \cite{fujita}).
\begin{prop}[\cite{kmm}, Theorem 1-3-1]
 Let $f:Y \to X$ be a proper birational morphism from a non-singular variety $Y$ onto a variety $X$ with divisors $L, \tilde L$ on $Y$.
 Assume that there exist $\mathbf Q$-divisors $D, \tilde D$ and an effective divisor $E$ 
 on $Y$ such that the following conditions are satisfied:
 \begin{enumerate}
 \item $\operatorname{supp} D$ and $\operatorname{supp} \tilde D$ are divisors with only simple normal crossings and  $\lfloor D \rfloor =\lfloor \tilde D \rfloor =0$;
 \item both $-L-D$ and $-\tilde L-\tilde D$ are $f$-nef;
 \item $K_Y\sim L+\tilde L+E$; and
 \item $E$ is  exceptional for $f$.
 \end{enumerate}
 Then, $R^if_*\o_Y(L)=0$ for $i>0$.
 \end{prop} 

\noindent
{\it Proof of Corollary \ref{j-canonical}}. 
Let $\nu: X^{\nu}\to X$ be the normalization and $f:Y\to X$ be a log resolution of 
$\j ac_X$.
By the assumption $\hk_{Y/X}-J_{Y/X}\geq 0$, we obtain
$$\hj(X, \o_X)=f_*\o_Y(\hk_{Y/X}-J_{Y/X})\supseteq f_*\cO_Y= \nu_*\o_{X^{\nu}}.$$
Since $\hj(X, \o_X)$ is an ideal of $\o_X$, we obtain $X^{\nu}=X$.
Now, we will prove $R^if_*\o_Y=0$ for $i>0$.
Note that $\hk_{Y/X}-J_{Y/X}=K_Y-B$ with $-B$ being $f$-nef, which we obtained in the proof of 
Theorem \ref{vanishing}.
Let $D=\tilde D=L=0$, $E=\hk_{Y/X}-J_{Y/X}$ and $\tilde L=B$,
then  $-L-D=0$ and  $-\tilde L -\tilde D=-B$ are both $f$-nef. 
The condition $K_Y\sim L+\tilde L+E$ is also satisfied.
Then, we can apply Fujita's vanishing theorem and obtain $R^if_*\o_Y=0$ for $i>0$.
$\Box$

\vskip.5truecm
We apply the above Local Vanishing Theorem (Theorem \ref{vanishing}) to prove two more restriction results. 
We say that a proper subscheme $X$ of  a variety $B$ is a hyperplane section if $B$ can be embedded
as a closed subvariety of a smooth variety $A$, and if there is a smooth hypersurface $H$ in $A$ such that $X=B\cap H$.

\begin{thm}[Restriction theorem for hyperplane sections]
\label{hyper}
Let $X$ be a closed subvariety of a variety $B$ which is a hyperplane section,
such that $X\cap B_{\rm reg}\neq\emptyset$. If
  $\tilde\a\subseteq \o_B$ is an ideal on $B$ such that its restriction
   $\a=\tilde\a\cdot \o_X$  is non-zero, then 
 for every $t\in {\mathbf R}_{\geq 0}$ we have
  $$\hj(X, \a^t)\subseteq \hj(B,\tilde\a^t)\cdot\cO_X.$$
\end{thm}

\begin{proof}
Let $B\hookrightarrow A$ be a closed embedding in a smooth variety, and $H$ a smooth hypersurface
in $A$ such that $X=B\cap H$. We denote by
$\tilde{\tilde\a}$ and $\b_B$ the pull back of $\tilde\a$ and $\j ac_B$, respectively, by the canonical projection $\o_A\to \o_B$, and by
 $\b_X$  the pull back of $\j ac_X$ by the canonical projection $\o_A\to \o_X$.
 
 We claim that there is a projective, birational morphism $\varphi_A\colon \overline A\to A$ that is an isomorphism at the generic points of $B$, $H$ and $X$, such that if 
$\overline B$, $\overline H$ and
$\overline X$ denote the proper transforms in $\overline{A}$ of $B$, $H$ and $X$, respectively,  
then the following hold:
\begin{enumerate}
\item[(i)]
$\varphi_A$ is a log resolution of $\tilde{\tilde\a}\cdot\b_B\cdot\b_X$,
\item[(ii)]
$\varphi_A$ is a factorizing resolution of $B$, and also of $H$, inside $A$,
\item[(iii)]
$\varphi_H:=\varphi_A|_{\overline H}\colon\overline H\to H$ is a factorizing resolution of $X$ inside $H$ and
$\varphi_B:=\varphi_A|_{\overline B}\colon\overline B\to B$ is a factorizing resolution of $X$ inside $B$,
\item[(iv)]
$\varphi_B$ and $\varphi_X:=\varphi_A|_{\overline X}\colon \overline X\to X$ 
are log resolutions of $\tilde\a\cdot {\j ac}_B$ and $\a\cdot {\j ac}_X$, respectively.
\end{enumerate}
We construct $\varphi_A$ as follows. We first take a log resolution
$\varphi'\colon A'\to A$ for the product of $\tilde{\tilde\a}$, $\b_B$, $\b_X$, and $I_{H/A}$, where 
$I_{H/A}$
is the defining ideal of $H$ in $A$. Since $X\cap B_{\rm reg}\neq\emptyset$, we may and will
assume that $\varphi'$ is an isomorphism
 at the generic points of $B$, $H$ and $X$. Let $E$ be the simple normal crossings divisor
 on $A'$ such that
$\o_{A'}(-E)=\tilde{\tilde\a}\cdot \b_B\cdot\b_X\cdot I_{H/A}\cdot\o_{A'}$. We now apply 
Lemma~\ref{singfac} to ($A$, $A'$,  $E$, $B$) to obtain a factorizing resolution $\varphi_A\colon\overline A \to A' \to A$
of $B$ inside $A$.
Note that $\varphi_A$ is also a factorizing resolution of $H$ inside $A$, since it is an embedded 
resolution of $(A, H)$ and the factorizing condition (3) in Definition~\ref{fact} is automatic, as 
$H$ is a hypersurface. 
Therefore $\varphi_A$ satisfies condition (ii) above, and 
conditions (i) and (iv) also follow by construction.
As $\varphi_A$ is a factorizing resolution of $B$ inside $A$, we have
$$I_{B/A}\cdot \o_{\overline A}=
I_{\overline B/\overline A}\cdot \o_{\overline A}(-R_{\overline B/\overline A})$$
for an effective divisor $R_{\overline B/\overline A}$ on $\overline A$, where
$I_{B/A}$ and  $I_{\overline B/\overline A}$ are the defining ideals of $B$ in $A$ and 
$\overline B$ in $\overline A$, respectively.
It is clear that $\overline{X}$ is a connected component of $\overline{B}\cap\overline{H}$.
After blowing-up the other components, we may assume that $\overline{X}=
\overline{B}\cap\overline{H}$.
Since $I_{X/H}=I_{B/A}\cdot\o_H$ and $I_{\overline X/\overline H}=I_{\overline B/\overline A}\cdot\o_{\overline H}$, we obtain
$$I_{X/H}\cdot\o_{\overline H}=I_{\overline X/\overline H}\cdot\o_{\overline H}(-R_{B/A}|_{\overline H}),$$
which shows the first part of (iii), and that we may take
$R_{X/H}=R_{B/A}|_{\overline H}$.
Using similarly that $I_{X/B}=I_{H/A}\cdot\o_B$ and $I_{\overline X/\overline B}=I_{\overline H/\overline A}\cdot \o_{\overline B}$, we obtain
$$I_{X/B}\cdot\o_{\overline B}=I_{\overline X/\overline B}\cdot\o_{\overline B}(-R_{H/A}|_{\overline B}).$$
This gives the second part of (iii), and that
we may take
$R_{X/B}=R_{H/A}|_{\overline B}$. Therefore $\varphi_A$ satisfies all our requirements.

Since $\varphi_A$ is a factorizing resolution of $B$, Lemma \ref{adjunction} gives
\begin{equation}
\label{eq1}
(K_{{\overline A}/A}-cR_{B/A})|_{\overline B}=\hk_{{\overline B}/B}-J_{{\overline B}/B},
\end{equation}
where $c=\codim(B, A)$.
As $\varphi_A$ is a factorizing resolution of the smooth hypersurface $H$, 
 Lemma \ref{adjunction} implies
\begin{equation}
\label{eq2}
(K_{{\overline A}/A}-R_{H/A})|_{\overline H}=K_{{\overline H}/H}.
\end{equation}
Finally, since $\varphi_H$ is a factorizing resolution of $X$ inside $H$, and we have  ${\rm codim}(X,H)=c$,  Lemma \ref{adjunction} gives
\begin{equation}
\label{eq3}
(K_{{\overline H}/H}-cR_{X/H})|_{\overline X}=\hk_{{\overline X}/X}-J_{{\overline X}/X}.
\end{equation}
Using (\ref{eq1}), (\ref{eq2}), (\ref{eq3}), and the formulas for $R_{X/H}$ and $R_{X/B}$,
we obtain 
\begin{equation}
\label{strongadjunction}
(\hk_{\overline B/B}-J_{\overline B/B}-R_{X/B})|_{\overline X}=\hk_{\overline X/X}-J_{\overline X/X}.
\end{equation}
Let $\tilde\a\cdot \o_{{\overline B}}=\o_{{\overline B}}(-Z_{{\overline B}/B})$ and $D=\hk_{{\overline B}/B}-J_{{\overline B}/B}-\lfloor tZ_{{\overline B}/B}\rfloor-R_{X/B}$.
We have an exact sequence:
\begin{equation}
\label{exact2} 
0\to {\varphi_{B}}_*( I_{\overline X/\overline B}\cdot\o_{\overline B}(D))\to {\varphi_{B}}_*( \o_{\overline B}(D))\to {\varphi_{X}}_*( \o_{\overline X}(D|_{\overline X}))
\end{equation}
$$ \to {R^1\varphi_{B}}_*( I_{\overline X/\overline B}\cdot\o_{\overline B}(D))
.$$
Arguing as in the proof of Proposition \ref{smoothrestriction}, we see that
the last term in this sequence is zero: instead of using the Local Vanishing Theorem for
multiplier ideals, we use 
Theorem \ref{vanishing} (in the version mentioned in Remark~\ref{rem_vanishing}).
Since $R_{X/B}$ is effective, we have an inclusion 
${\varphi_{B}}_*( \o_{\overline B}(D))\subseteq \hj (B, \tilde\a^t)$,
hence the exact sequence (\ref{exact2}) gives
$$ {\varphi_{X}}_*( \o_{\overline X}(D|_{\overline X}))\subseteq \hj (B, \tilde\a^t)\cdot\cO_X.$$
On the other hand, it follows from (\ref{strongadjunction}) that
$${\varphi_{X}}_*( \o_{\overline X}(D|_{\overline X}))={\varphi_{X}}_*\left( \o_{\overline X}(\hk_{{\overline X}/X}
-J_{{\overline X}/X}-\lfloor tZ_{{\overline X}/X}\rfloor)\right)=\hj (X, \a^t),$$
and this completes the proof of the theorem.
\end{proof}

\begin{cor}[Restriction theorem for Cohen-Macaulay varieties with proper intersection] 
\label{cm}
Let $B$ and $H$ be Cohen-Macaulay subvarieties of a smooth ambient variety $A$,
such that $X=B\cap H$ is irreducible and reduced, and the intersection is proper, that is,
$\dim(X)=\dim(B)+\dim(H)-\dim(A)$.
If $X\cap B_{\rm reg}\neq\emptyset$, and
$\tilde\a\subset \o_B$ is an ideal on $B$ whose restriction $\a=\tilde\a\cdot \o_X$  to $X$ is nonzero, then for every $t\in {\mathbf R}_{\geq 0}$ we have
\begin{equation}\label{eq_cm}
\hj(X, \a^t)\subseteq \hj(B,\tilde\a^t)\cdot\cO_X.
\end{equation}
\end{cor}

\begin{proof}
Consider the diagonal $\Delta\subset A\times A$.
Since $H$ and $B$ intersect properly, it follows that
  $X\simeq \Delta\cap (H\times B)$ is of codimension $N=\dim(A)$ in $H\times B$.
Both $A\times A$ and $\Delta$ are smooth, hence $\Delta$ is locally defined in $A\times A$
by $N$ equations.  
Since the inclusion (\ref{eq_cm}) is a local statement, we may replace $A$ by a suitable open cover, and therefore we may assume that $A$ is affine, and the ideal of $\Delta$ in 
$A\times A$ is generated by $N$ equations $f_1,\ldots,f_N\in\cO(A\times A)$. 
  
  We claim that  there is 
  a sequence of varieties
  $$X=X_0\subset X_1\subset\cdots \subset X_N=H\times B$$
  such that   $X_i$ is a hyperplane section of $X_{i+1}$ for each $i$ with $0\leq i< N$.
  Let $H_1,\ldots,H_N$ be general elements of the linear system 
  $\Lambda=\{\sum _i \lambda_if_i\mid
  \lambda_i\in k\}$, and let $X_i=(H\times B)\cap H_1\cap\ldots\cap H_{N-i}$. 
  Note that $X_0=X$, and we make the convention $X_N=H\times B$.
  By Bertini's Theorem, each $H_1\cap\ldots\cap H_i$ is smooth, hence $X_i$
  is a hyperplane section in $X_{i+1}$ for $i<N$. Furthermore, $H\times B$
  is Cohen-Macaulay by assumption, hence each $X_i$ is Cohen-Macaulay. In particular,
  no $X_i$ has embedded components. We now prove by descending induction
  that $X_i$ is irreducible and reduced for every $i$ with $1\leq i\leq N$. 
  The case $i=N$ is clear, and let us assume that $X_{i+1}$ is irreducible and reduced for some
  $i\geq 1$. The base locus of
  $\Lambda\vert_{H\times B}$ is equal to $X$, hence another application of Bertini's theorem
  implies that $X_i\smallsetminus X$ is irreducible and reduced.
  Since $X$ is not an irreducible component of $X_i$ by dimension considerations, it follows that
  $X_i$ is irreducible and generically reduced, hence reduced (recall that $X_i$ has no embedded components). This completes the proof of our claim. 
  
  Let $p_2\colon H\times B\to B$ be the canonical projection.
Applying successively Theorem \ref{hyper} for the above sequence of hyperplane sections, we obtain
  \begin{equation}\label{eq2_cm}
  \hj(X,\a^t)\subseteq \hj(H\times B, p_2^{*}(\tilde\a)^t)\cdot\o_X.
  \end{equation}
On the other hand, if we consider 
an appropriate resolution of $H\times B$ of the form 
   $\overline H\times \overline B\to H\times B$, and if we denote by
$q_1\colon\overline{H}\times\overline{B}\to\overline{H}$ and $q_2\colon
\overline{H}\times\overline{B}\to\overline{B}$ the two projections, then we have   
   $$\hk_{(\overline H\times \overline B)/(H\times B)}=
   q_1^*\hk_{\overline H/H}+q_2^*\hk_{\overline B/B},\,\,
   J_{(\overline H\times \overline B)/(H\times B)}=
   q_1^*J_{\overline H/H}+q_2^*J_{\overline B/B}.$$
Using also the K\"{u}nneth formula, these imply
 $$\hj(H\times B, q_2^{*}(\tilde\a)^t)=q_1^{*}\hj(H, \o_H)\otimes q_2^{*}\hj(B, \tilde\a^t),$$
 and we thus obtain
 $$\hj(H\times B, p_2^{*}(\tilde\a)^t)\cdot\o_X\subseteq (\hj(H, \o_H)\cdot\o_X)\cdot
 (q_2^*\hj(B, \tilde\a^t)\cdot\o_X)
   \subseteq \hj(B, \tilde\a^t)\cdot\o_X.$$
Combining this with (\ref{eq2_cm}), we obtain (\ref{eq_cm}).
\end{proof}

\begin{question}
Does the assertion in Corollary~\ref{cm} hold if we drop the assumption that
$H$ and $B$ are Cohen-Macaulay?
\end{question}

On the other hand, the following example shows that one does not have a restriction theorem
when $X$ is an arbitrary subvariety of a singular\footnote{When the ambient variety $B$ is smooth, the restriction theorem is provided by Proposition~\ref{smoothrestriction}.} variety $B$.
\begin{exmp}
 Let $B$ be the hypersurface defined by $x_1^d+\cdots+x_n^d=0$ in ${\mathbf A}^n$, with 
 $d\geq n\geq 3$.
 Note that $(B,0)$ is a non-canonical normal isolated singularity.
 $B$ is a complete intersection, hence
 $$\hj (B, \o_B)=\j(B, \o_B)=f_*(K_{Y/X}),$$ where $f\colon Y\to X$ is a log resolution of 
 $\j ac_B$. Since $B$ is non-canonical, we see that $\hj (B, \o_B)$ is contained in the maximal
 ideal defining $0$.
Suppose now that $X\subset B$ is a line through $0$. 
 Since $X$ is smooth, we have $\hj (X, \o_X)=\j(X, \o_X)=\o_X$.
 Therefore the inclusion $\hj(X, \o_X)\subseteq \hj(B,\o_B)_X$ does not hold.

\end{exmp}

We can now apply our vanishing result in Theorem~\ref{vanishing} to deduce a version
of the Subadditivity Theorem for Mather multiplier ideals. We refer to \cite[\S 9.5.B]{laz} and
\cite[Theorem~6.5]{ee}
for the corresponding result in the setting of usual multiplier ideals. In particular, our approach follows closely
the one in \cite{ee}, which treats the case of singular varieties. 
We start with the 
following adjunction-type formula, which is
\cite[Lemma~6.3]{ee}.

\begin{lem}
\label{diagonaladjunction}
Let $X$ be an $n$-dimensional variety and $Y\to X$ a log resolution of $\j ac_X$.
Consider 
$A=X\times X$ and $A'=Y\times Y$, and let
$\varphi\colon \overline A\to A' \to A$ be a factorizing resolution of the diagonal $\Delta$ inside $A$,
provided by Lemma \ref{singfac}.
If $\overline \Delta$ is the strict transform of $\Delta$ in $\overline{A}$, 
then we have the following inequality of divisors on $\overline A$:
$$\left(\hk_{\overline A/A}-J_{\overline A/A}-nR_{\Delta/A}\right)|_{\overline \Delta}
\geq \hk_{\overline\Delta/\Delta}-2J_{\overline\Delta/\Delta}.$$
\end{lem}

\begin{thm}
\label{subad}
 If $\a$ and $\b$ are nonzero ideals on a variety $X$, then for every $s,t\in{\mathbf R}_{\geq 0}$
 we have
 $$\hj(X, \j ac_X\a^s\b^t)\subseteq \hj(X, \a^s)\cdot\hj(X, \b^t).$$
\end{thm}

\begin{proof}
  Let $g\colon Y\to X$ be a log resolution of $\j ac_X\cdot \a\cdot \b$. We write
  $\a\cdot \o_Y=\o_Y(-Z_{Y/X})$ and $\b\cdot \o_Y=\o_Y(-W_{Y/X})$.
  Let $g\times g\colon A'=Y\times Y\to A:=  X\times X$ be the product map. We denote by   
  $\tilde p$ and $\tilde q$ 
  the first and the second projections from $Y\times Y$ to $Y$, and by 
  $p$ and $q$ the first and the second projections from $X \times X$ to $X$.
 Let  $E$ be the divisor on $A'$ such that 
 $\o_{A'}(-E)=p^{*}(\a\cdot \j ac_X )\cdot q^{*}(\b\cdot \j ac_X)$. We consider
 a factorizing resolution $\varphi\colon \overline A\to A' \to A$  of the diagonal $\Delta$ in $A$,
obtained by applying Lemma~\ref{singfac} with respect to $E$.
Let $\overline \Delta$ be the strict transform of $\Delta$ in $\overline A$, and 
let $\varphi_{\Delta}=\varphi|_{\overline\Delta}$.
If we write $\a\cdot \cO_{\overline{\Delta}}=\cO_{\overline{\Delta}}(-Z_{\overline\Delta/\Delta})$
and $\b\cdot \cO_{\overline{\Delta}}=\cO_{\overline{\Delta}}(-W_{\overline\Delta/\Delta})$, then
after identifying
$X$ and $\Delta$, we obtain
$$\hj(X,\j ac_X\a^s\b^t)={\varphi_{\Delta}}_*\left(\hk_{\overline\Delta/\Delta}-2J_{\overline\Delta/\Delta}
-\lfloor sZ_{\overline\Delta/\Delta}\rfloor-\lfloor tW_{\overline\Delta/\Delta}\rfloor\right),$$
and by Lemma~\ref{diagonaladjunction},
 the right hand side is contained in
\begin{equation}\label{eq_subadditivity}
{\varphi_{\Delta}}_*\left(\hk_{\overline A/A}-J_{\overline A/A}
-\lfloor s\tilde  p^*Z_{\overline Y/X}\rfloor-\lfloor t\tilde q^*W_{\overline Y/X}\rfloor-nR_{\Delta/A}|_{\overline\Delta}\right),
\end{equation}
where $n=\dim(X)$. Arguing as in the proof of Theorem \ref{hyper}, we see using 
Theorem \ref{vanishing} that the ideal in (\ref{eq_subadditivity}) is contained in 
$$\hj (X\times X, (p^*\a)^s\cdot (q^*\b)^t)\cdot\cO_{\Delta}.$$
Considering the log resolution $g\times g\colon A'\to A$  of
${\j ac}_A\cdot p^*(\a)\cdot q^*(\b)$, and
noting that $\hk_{A'/A}=\tilde p^*\hk_{Y/X}+\tilde q^*\hk_{Y/X}$ and 
$J_{A'/A}=\tilde p^*J_{Y/X}+\tilde q^*J_{Y/X}$,
we obtain 
$$\hj (X\times X, (p^*\a)^s\cdot (q^*\b)^t)\cdot \cO_{\Delta}=\hj(X,\a^s)\cdot\hj(X,\b^t),$$
and we obtain the inclusion in the statement of the theorem.
 \end{proof}
 
%

We now give a Skoda-type formula for Mather multiplier ideas (see 
\cite[\S 11.1.A]{laz} for the statement in the context of usual multiplier ideals).
The proof follows the one in \emph{loc. cit.}, using our vanishing result in Theorem~\ref{vanishing}.

\begin{thm}
\label{skoda}
If $X$ is a variety of dimension $n$, and $\a\subseteq \o_X$  is a nonzero ideal on $X$, then
for every $t\geq n$ we have
$$\hj(X, \a^t)=\a\cdot \hj(X, \a^{t-1}).$$
In particular, if $t$ is an integer, then
$$\hj(X, \a^t)=\a^{t-n+1}\cdot \hj(X, \a^{n-1})$$
\end{thm}

\begin{proof}
The inclusion $\a\cdot \hj(X, \a^{t-1})\subseteq \hj(X, \a^t)$ holds for every $t\geq 1$, and it is an immediate consequence of the definition of the Mather multiplier ideal. Therefore from now on we focus on the reverse inclusion.

  If $X$ is covered by open subsets $U_i$, then it is clearly enough to prove the theorem for each $U_i$. Since $X$ is $n$-dimensional, after taking a suitable open cover, we may assume that
$X$ is affine, and there is a reduction $\b$ of $\a$ generated by $n$ elements 
$s_1,\ldots,s_n$ (see \cite[Proposition~4.6.8]{bh}). This means that $\b\subseteq\a$, and there
is $m\geq 1$ such that $\b\a^m=\a^{m+1}$. 

 Let $\varphi\colon Y\to X$ be a log-resolution of $\a\cdot\b\cdot \j ac_X$. 
 We write $\a\cdot\o_Y=\b\cdot\o_Y=\o_Y(-Z_{Y/X})$ and
consider ${\mathcal L}_s=\o_Y(\hk_{Y/X}-J_{Y/X}-\lfloor sZ_{Y/X}\rfloor)$ for every 
$s\in {\mathbf R}_{\geq 0}$.
It follows from Theorem \ref{vanishing} that $R^p\varphi_*{\mathcal L}_s=0$ for $p> 0$.

If $V$ is the $k$-vector space generated by $s_1,\ldots,s_n$, then we have on $Y$
an exact Koszul complex
$$0\to \wedge^nV\otimes_k\o_Y(nZ_{Y/X})\to \cdots\to  V\otimes_k\o_Y(Z_{Y/X})\to \o_Y\to 0.$$
By tensoring with ${\mathcal L}_t$ we get the exact complex
$$0\to  \wedge^nV\otimes_k{\mathcal L}_{t-n}\to\cdots\to V\otimes_k{\mathcal L}_{t-1}
\to {\mathcal L}_{t}\to 0.$$
Since $R^p\varphi_*{\mathcal L}_{t-i}=0$ for $0\leq i\leq n$ and for all $p>0$, we deduce
that we have an exact complex on $X$
$$0\to  \wedge^nV\otimes_k\varphi_*{\mathcal L}_{t-n}\to\cdots\to V\otimes_k\varphi_*{\mathcal L}_{t-1}
\to \varphi_*{\mathcal L}_{t}\to 0.$$
In particular, the map $V\otimes_k\varphi_*{\mathcal L}_{t-1}
\to \varphi_*{\mathcal L}_{t}$ is surjective, hence
$$\hj(X, \a^{t})=\b\cdot \hj(X, \a^{t-1})\subseteq
\a\cdot\hj(X, \a^{t-1}).$$
This completes the proof of the theorem.
\end{proof}

\begin{cor}[Brian\cedilla{c}on-Skoda type formula]
If $\a$ is an ideal on an $n$-dimensional variety $X$, then for every integer $m\geq n$ we have
$$\overline{{\j ac}_X\cdot\a^m}\subseteq \a^{m-n+1}.$$
\end{cor}

\begin{proof}
We may assume that $\a$ is nonzero, since otherwise the assertion is trivial.
We keep the notation in the proof of Theorem~\ref{skoda}. 
Since the divisor $\hk_{Y/X}$ is effective, we have an inclusion
$$\o_Y(-J_{Y/X}-mZ_{Y/X})\subseteq \o_Y(\hk_{Y/X}-J_{Y/X}-mZ_{Y/X}).$$ 
We deduce
$$\overline{\j ac_X\cdot \a^m}\subseteq \varphi_*(\o_Y(-J_{Y/X}-mZ_{Y/X}))$$
$$\subseteq \varphi_*(\o_Y(\hk_{Y/X}-J_{Y/X}-mZ_{Y/X}))=\hj(X, \a^m).$$
Using Theorem \ref{skoda}, we get
$$\overline{\j ac_X\cdot \a^m}\subseteq\hj(X, \a^m)=\a^{m-n+1}\cdot\hj(X, \a^{n-1})\subseteq
\a^{m-n+1}.$$
\end{proof}

We also have a version of the Summation Theorem (see \cite[Corollary~2]{JM})
for Mather multiplier ideals.

\begin{thm}\label{summation}
If $\a_1,\ldots,\a_r,\b$ are nonzero ideals on the variety $X$, then
for every $s,t\in {\mathbf R}_{\geq 0}$ we have
\begin{equation}\label{eq_summation}
\hj\left(X, (\a_1+\ldots+\a_r)^s\b^t\right)= \sum_{s_1+\ldots+s_r=s}
\hj(X,\a_1^{s_1}\cdots\a_r^{s_r}\b^t).
\end{equation}
\end{thm}

\begin{proof}
Let $\varphi\colon Y\to X$ be a log resolution of the product of $\b$, ${\j ac}_X$, and of all
$\a_i+\a_j$, for $1\leq i,j\leq r$. It follows from \cite[Lemma~3.1]{JM}
that $\varphi$ is also a log resolution of $\b\cdot (\a_1+\ldots+\a_r)\cdot {\j ac}_X$. 
Furthermore, if we write $\a_i\cdot\cO_Y=\cO_Y(-A_i)$, $\b\cdot \cO_Y=\cO_Y(-B)$, and
$(\a_1+\ldots+\a_r)\cdot\cO_Y=\cO_Y(-A)$, then $A$ is the largest effective divisor on $Y$ with
$A\leq A_i$ for all $i$. 
For every $s_1,\ldots,s_r$ with
$\sum_is_i=s$, we have
$$\lfloor s_1A_1+\ldots+s_rA_r\rfloor\geq\sum_{i=1}^r\lfloor sA\rfloor,$$
and the inclusion ``$\supseteq$" in (\ref{eq_summation}) then follows from definition
of Mather multiplier ideals. 

We now consider the reverse inclusion. It is shown in \cite[\S 3]{JM} that there is an exact complex on $Y$
$${\mathcal E}_{\bullet}:\,\,\,
0\to {\mathcal E}_{r-1}\to\ldots\to{\mathcal E}_1
\to {\mathcal E}_0\to 0$$
with the following properties:
\begin{enumerate}
\item ${\mathcal E}_0=\cO_Y(-A)$.
\item ${\mathcal E}_1$ is a finite direct sum of sheaves of the form
$\cO_Y(-\lfloor s_1A_1+\ldots+s_rA_r+tB\rfloor)$, with $\sum_is_i=s$.
\item Every ${\mathcal E}_i$ with $2\leq i\leq r-1$ is a finite direct sum of 
sheaves of the form $\cO_Y(-\lfloor a_1A_1+\ldots+a_rA_r+bB\rfloor)$, for some
$a_1,\ldots,a_r,b\in{\mathbf R}_{\geq 0}$. 
\end{enumerate}

It follows from the above properties and Theorem~\ref{vanishing} that 
$$R^i\varphi_*({\mathcal E}_j\otimes\cO_Y(\widehat{K}_{Y/X}-J_{Y/X}))=0\,\,\text{for all} \,\,i>0\,\,\text{and}\,\,j\geq 0.$$
 This implies that the push-forward $\varphi_*({\mathcal E}_{\bullet}\otimes\cO_Y(\widehat{K}_{Y/X}-J_{Y/X}))$ is exact. The surjectivity of the last map
in this complex, and properties (1) and (2) give the inclusion
``$\subseteq$" in (\ref{eq_summation}).
\end{proof}

We end with some considerations concerning an asymptotic version of Mather multiplier
ideals. For basic facts about graded sequences of ideals and their multiplier ideals, we refer to
\cite[Chapter~11]{laz}. Let $X$ be a variety, and let $\a_{\bullet}=(\a_m)_{m\geq 1}$
be a graded sequence of ideals on $X$, that is, $\a_p\cdot\a_q\subseteq\a_{p+q}$ for
all $p,q\geq 1$. We assume that some $\a_p$ is nonzero. 
Given $t\in{\mathbf R}_{\geq 0}$ and $p,m\geq 1$ such that
$\a_p$ is nonzero, we obtain using Proposition~\ref{elementary}
$$\hj(X,\a_p^{t/p})=\hj(X, (\a_p^m)^{t/mp})\subseteq\hj(X,\a_{pm}^{t/mp}).$$
It follows from the Noetherian property that there is an ideal $\hj(X,\a_{\bullet}^t)$ such that
$\hj(X,\a_p^{t/p})=\hj(X,\a_{\bullet}^t)$ whenever $p$ is divisible enough. This is the
\emph{asymptotic Mather multiplier ideal} of $\a_{\bullet}$. 
It is straightforward to extend Theorems~\ref{hyper}, \ref{cm}, \ref{subad}, and 
\ref{summation} to the setting of graded sequences of ideals. 
We leave for the reader the task of formulating these extensions.
We only note that if
$\a^{(1)}_{\bullet},\ldots,\a^{(r)}_{\bullet}$ are graded sequences of ideals, then their product and sum have as their $p^{\rm th}$ terms $\prod_{i=1}^r\a^{(i)}_p$, and respectively,
$\sum_{i_1+\ldots+i_r=p}\a^{(1)}_{i_1}\cdot\ldots\cdot\a^{(r)}_{i_r}$.

\makeatletter \renewcommand{\@biblabel}[1]{\hfill#1.}\makeatother

\end{document}